\documentclass{amsart}
\usepackage{amssymb,bbm,color,xcolor}
\usepackage{amscd}
\usepackage{enumerate}
\usepackage{siunitx}
\usepackage{bm}
\usepackage{ulem}
\usepackage{float, graphicx}%图片浮动，插图包
\usepackage{subfigure}
\numberwithin{equation}{section}
\usepackage[pagebackref]{hyperref}

\usepackage{mathtools}%                  http://www.ctan.org/pkg/mathtools
\usepackage[tableposition=top]{caption}% http://www.ctan.org/pkg/caption
\usepackage{booktabs,dcolumn}%           http://www.ctan.org/pkg/dcolumn + http://www.ctan.org/pkg/booktabs

\DeclareFontFamily{OT1}{rsfs}{}
\DeclareFontShape{OT1}{rsfs}{n}{it}{<-> rsfs10}{}
\DeclareMathAlphabet{\mathscr}{OT1}{rsfs}{n}{it}

\theoremstyle{plain}

\newtheorem{theorem}{Theorem}[section]

\newtheorem{lem}[theorem]{Lemma}

\theoremstyle{definition}

\newtheorem{remark}[theorem]{Remark}

\newcommand{\bal}{\[\begin{aligned}}
\newcommand{\eal}{\end{aligned}\]}
\newcommand{\beeq}{\begin{equation}}\newcommand{\eneq}{\end{equation}}
\newcommand{\beqa}{\begin{eqnarray*}}\newcommand{\eeqa}{\end{eqnarray*}}
\newcommand{\beeqa}{\begin{eqnarray}}\newcommand{\eneqa}{\end{eqnarray}}

\def\<{\langle}             \def\>{\rangle}
\newcommand{\al}{\alpha}    \newcommand{\be}{\beta}
  \newcommand{\vep}{\varepsilon}
  \newcommand{\La}{\Lambda}
    \newcommand{\la}{\lambda}
\newcommand{\sig}{\sigma}  
  \newcommand{\Si}{\Sigma}
\newcommand{\om}{\omega}    
\newcommand{\gam}{\gamma}
\newcommand{\R}{\mathbb{R}}

\newcommand{\Sp}{\mathbb{S}}
\newcommand{\ms}{\mathbb{S}}
\newcommand{\pt}{\partial_t}\newcommand{\pa}{\partial}
\newcommand{\les}{{\lesssim}}
\newcommand{\pri}{\prime}

\newcommand{\hn}{\mathbb{H}^n}

\newcommand{\htw}{\mathbb{H}^2}

\newcommand{\De}{\Delta}
\newcommand{\de}{\delta}

\newcommand{\hf}{\frac{1}{2}}

\newcommand{\one}{\uppercase\expandafter{\romannumeral1}}
\newcommand{\two}{\uppercase\expandafter{\romannumeral2}}
\newcommand{\three}{\uppercase\expandafter{\romannumeral3}}
\newcommand{\four}{\uppercase\expandafter{\romannumeral4}}
\newcommand{\five}{\uppercase\expandafter{\romannumeral5}}
\newcommand{\six}{\uppercase\expandafter{\romannumeral6}}
\newcommand{\sev}{\uppercase\expandafter{\romannumeral7}}
\newcommand{\eig}{\uppercase\expandafter{\romannumeral8}}
\newcommand{\wt}{\widetilde}

\title
[Wave equations with logarithmic nonlinearity]{
The Critical Conjecture of Wave Equations with Logarithmic Nonlinearity on $\mathbb{H}^2$
}
\author{Xiaoran Zhang}%\thanks{* Corresponding author}
\address{School of Mathematical Sciences\\ Zhejiang University\\ Hangzhou 310058,P.R.China}
\email{1025391337@qq.com}

\thanks{The author was supported by
 NSFC 11971428  and  NSFC 12141102. }

\date{\today}
\begin{document}

\begin{abstract}
In this paper, we verified our critical conjecture in \cite{wang2023wave} on two-dimensional hyperbolic space, that is, 
concerning nonlinear wave equations with logarithmic nonlinearity, which behaves like $\left(\ln {1}/{|u|}\right)^{1-p}|u|$ near $u=0$, on hyperbolic spaces, we demonstrate that the critical power is $p_c(2)=3$, by proving global existence for $p>3$, as well as blow up for $p\in (1,3)$.

%In light of the exponential decay of solutions of linear wave equations on hyperbolic spaces $\hn$, to illustrate the critical nature, we investigate nonlinear wave equations with logarithmic nonlinearity, which behaves like
%$\left(\ln {1}/{|u|}\right)^{1-p}|u|$ near $u=0$, 
%on hyperbolic spaces.
%Concerning the global existence vs blow up with small data, we expect that the problem admits a critical power $p_c(n)>1$.
%When $n=3$,
%we prove that the critical power is $3$, by proving global existence for $p>3$, as well as generically blow up for $p\in (1,3)$.
\end{abstract}


%%%%%%%%%%

\maketitle
%\tableofcontents

\section{Introduction}

Let $n\ge 2$, consider the wave equation
\beeq\label{nlw}
%\tag{NLW}
\begin{cases}
\partial_{t}^2 u-(\De_{\hn}+\rho^2)u=F(u)\ ,\\
u(0,x)=\vep u_0, u_t(0,x)=\vep u_1\ ,
\end{cases}
\eneq
where $\rho=\frac{n-1}{2}$ (recall that the spectrum of $-\De_{\hn}$ is $[\rho^2,\infty)$), $u_0, u_1$ are smooth functions with compact support. As is well known, the global existence vs blow-up for nonlinear wave equations with power-type nonlinearities $F(u)\sim |u|^p$ is related to the so-called $Strauss\ conjecture$ in $\R^n$, which has a critical power $p_c(n)>1$. Correspondingly on hyperbolic spaces  $\hn$, it is known to admit global solutions for sufficiently small $\vep>0$, for any power $p\in (1, 1+4/(n-1))$, thanks to the improved (exponential) %polynomial or exponentially
 decay of solutions of linear wave equations.  In some sense, in handling the power nonlinearity, we do not need to explore the precise information on the decay rate and no critical phenomenon appears.
Therefore, it is interesting to capture the critical nature by investigating proper nonlinearities, such as logarithmic form, in an attempt to convert the essence of exponential decay to polynomial decay.

%To capture the critical nature, in our previous paper \cite{wang2023wave}, we propose 
%the investigation of nonlinear wave equations with logarithmic nonlinearities $F_p(u)$ near $u=0$, for which we expect to have a critical power $p_c(n)>1$.
% want to explore the critical state of $F(u)$ similar to $Strauss\ conjecture$ in $\R^n$ by choosing logarithmic nonlinearities $F(u)$.
%small data blow-up for \eqref{nlw} 

The interest arises from the similar equation in Euclidean spaces
\beeq\label{esnlw}
\begin{cases}
\partial_{t}^2 u-\De_{\R^n}u= |u|^p\ ,\\
u(0,x)=\vep u_0,\ u_t(0,x)=\vep u_1.
\end{cases}
\eneq
It has been studied for a long time and admits a critical power $p=p_c(n)>1$, such that for any compactly supported initial data
with sufficiently small size ($\vep\ll 1$), a regular global solution exists when $p>p_c(n)$, while such a result fails when $1<p<p_c(n)$. The first work in this direction is F. John’s work \cite{MR535704} in 1979 
to determine the critical power $p_c(3)=1+\sqrt{2}$.
%when $n=3$, where F. John determined the critical power $p_c(3)=1+\sqrt{2}$. 
Then Strauss \cite{MR614228} conjectured that the critical power $p_c(n)$ for other dimensions $n\ge2$ should be the positive root of the quadratic equation
$$(n-1)p^2-(n+1)p-2=0.$$
The conjecture was verified in Glassey \cite{MR631631}, \cite{MR618199} when $n=2$ with $p_c(2)=(3+\sqrt{17})/2$,
and in \cite{MR1331521}, \cite{MR1408499}, \cite{MR1481816}, \cite{MR1804518}, \cite{MR744303} for other dimensions. There are many articles and reviews that list the history of the $Strauss\ conjecture$, 
such as \cite{wang2023wave}.
% Then for other dimensions, the existence portion of the conjecture was proved by Zhou \cite{MR1331521} ($n=4$), Lindblad-Sogge \cite{MR1408499} ($n\le8$) and Georgiev-Lindblad-Sogge \cite{MR1481816}, Tataru \cite{MR1804518} (all $n$, $p_c(n)<p\le p_{conf}$), where
%$$p_{conf}(n)=1+\frac{4}{n-1}$$
%is the conformal power. While the blow-up portion is due to Sideris \cite{MR744303} ($n\ge4$, $1<p<p_c(n)$).

%, we choose $n$ orthonomal vector fields $\partial_r ,\frac{\partial_{\om}}{\sinh r}$
On hyperbolic spaces $\hn$, if we take the same power-type nonlinearities $F(u)= |u|^p$, we expect that small data global existences always hold for all $p>1$, owing to the negative curvature. %More precisely, the metric in geodesic polar coordinates is given by $g_{\hn}=dr^2+(\sinh r)^2d\om^2$, for $(r,\omega)\in (0,\infty)\times\Sp^{n-1}$, see Section \ref{pre}.
%in geodesic polar coordinates, the metric is given by $g_{\hn}=dr^2+(\sinh r)^2d\om^2$, for $(r,\omega)\in (0,\infty)\times\Sp^{n-1}$,
% see Section \ref{pre}. If we take the same power-type nonlinearities $F(u)= |u|^p$, heuristically, we expect that small data global existences always hold for all $p>1$, 
% due to the $\sinh r$ factor in the metric. 
 It is first proved by Fontaine \cite{MR1443052} in 1997 when $n=2,3$, % from the perspective of Lie algebra, 
 for any data $u_0\in C^1(\hn),u_1\in C(\hn)$ satisfying
\beeqa\label{data}
|u_0|+|u_1|+|\nabla u_0|\le  \theta_k\ , 
\eneqa
where $\nabla f=(\partial_r f\partial_r,(\frac{1}{\sinh r})^2\partial^{\om}f\partial_{\om})$, 
$|\nabla f|=((\partial_r f)^2+|\frac{\partial_{\om}f}{\sinh r}|^2)^{1/2}$ and $\theta_k =(\cosh |\cdot|)^{-k-\rho}$, $k>0$, with the metric in geodesic polar coordinates $g_{\hn}=dr^2+(\sinh r)^2d\om^2$, for $(r,\omega)\in (0,\infty)\times\Sp^{n-1}$, see Section \ref{pre}.
For general spatial dimensions $n\ge 2$, we refer the reader to \cite{MR2902129}, \cite{MR4026182}, \cite{MR2775816}, \cite{MR2944727}, \cite{MR3254350} for shifted and nonshifted wave equations and also \cite{MR4169670}, \cite{MR4432952}, \cite{MR4167287}, \cite{MR4310334}, \cite{MR3345662} for similar results on related spaces.
%asymptotically hyperbolic manifolds, symmetric spaces and {D}amek-{R}icci spaces. 
%Anker-Pierfelice-Vallarino \cite{MR2902129} 
%proved the improved (polynomial) dispersive and Strichartz estimates, which is strong enough to imply global results for $1<p\le p_{conf}(n)$, even though such results have not been stated explicitly. Based on Tataru's (exponential) dispersive estimates \cite{MR1804518},
%the global results for  $1<p\le p_{conf}(n)$ were explicitly stated and proved by
% Sire-Sogge-Wang \cite{MR4026182}.
%The nonshifted wave equations (with 
%$\partial_{t}^2 u-\De_{\hn}$
%instead of
%$\partial_{t}^2 u-(\De_{\hn}+\rho^2)$)
% have also been investigated in
% Metcalfe-Taylor
%\cite{MR2775816}, \cite{MR2944727}, and Anker-Pierfelice \cite{MR3254350}. 
% See also Anker-Pierfelice-Vallarino
%\cite{MR3345662} for similar results on
% {D}amek-{R}icci spaces, Zhang \cite{MR4167287}, \cite{MR4310334} for similar results on symmetric spaces,
%  as well as the recent works of Sire-Sogge-Wang-Zhang \cite{MR4169670}, \cite{MR4432952} for similar results on asymptotically hyperbolic manifolds. 
  All these results show that the critical power is $p_c(n)=1$, or we can say that there is no critical powers on hyperbolic spaces with power-type nonlinearities.

%The authors want to point out that the assumptions of $\theta_k$ are similar to Asakura’s \cite{MR862696} if the initial data are not compactly supported, the comparison scale being defined here by functions related to the areas of spheres, namely $(\sinh r)^{n-1}$ on $\hn$.

%%%%%%确认近期没有新的相关文献出现

Based on the fact of the exponential decay of (linear) solutions which was explicitly stated in \cite{MR1443052}, \cite{MR1804518}
and \cite[Lemma 2.1]{wang2023wave},
%(see, e.g., \cite{MR1443052}, \cite{MR1804518} or \cite[Lemma 2.1]{wang2023wave}),
we proposed
the investigation of nonlinear wave equations with logarithmic nonlinearities $F_p(u)$ near $u=0$, for some $p>1$ in our previous paper \cite{wang2023wave}.
One typical example
is \beeq\label{eq-nlterm}
F_p(u)=\left(\sinh^{-1} \frac{1}{|u|}\right)^{-(p-1)}|u|\ ,\eneq 
which behaves like
 $(\ln\frac{1}{|u|})^{1-p}|u|$ for small $|u|$ and $|u|^{p}$ for large $|u|$. Further, for such nonlinearities, 
we stated our conjecture that  the critical power is $p_c(n)=3$, regardless of the spatial dimension. More precisely, we conjectured that there exists
$\de=\de(n)>0$ so that 
we have global existence, with small data, for any $p\in (3, 3+\de(n))$, while for
$p\in (1,3)$, there exist
 some data $(u_0, u_1)$ so that there is no global solutions for any $\vep>0$. In \cite{wang2023wave},
 we have proved the conjecture in $n=3$ with the help of the simple representation of solutions in physical case. As is well known, from the perspective of the expression of the solution, it is more difficult to give an argument for nonphysical case, especially for even dimensions.

%Actually,  we propose 
%the investigation of nonlinear wave equations with logarithmic nonlinearities $F_p(u)$ near $u=0$, for example \eqref{eq-nlterm},
%for which we conjecture the critical power $p_c(n)=3$ and we have proved this in $n=3$.

%Thanks to the $\sinh r$ factor in the metric,
% we expect exponential decay of (linear) solutions, see, e.g., \cite{MR1443052}, \cite{MR1804518} or
%  Lemma \ref{formulae}. More precisely, by \eqref{radialsolu} in the appendix,
%we are convinced that,
%for smooth data with compact support,
% the linear solution behaves like $(\sinh t)^{-\rho}\sim e^{-\rho t}$ near the light cone $t=r$ as $t$ goes to infinity, at least when $n=3$. 
%In light of the exponential decay of linear solutions, to illustrate the critical nature, 
%it is natural to introduce the logarithmic nonlinearity, 
% which behaves like $\left(\ln {1}/{|u|}\right)^{1-p}|u|$ near $u=0$, 

%Concerning the problem of
%global existence vs blow up for the
% Cauchy problem \eqref{nlw} with $F=F_p(u)$,
%we expect there exist a critical power $p_c(n)>1$ and it is interesting to  determine the critical power $p_c(n)$ for any $n\ge 2$.

% and leave other dimensions for future investigation

In this paper, we will concentrate on  $n=2$. 
At first, concerning the problem of global existence with small data, we need only to assume the behavior of $F_p$ near $0$, that is, 
$F_p\in C^1$, $F_p(0)=F'_p(0)=0$, and
\beeqa\label{nonlinear}
|F_p'(u)|\lesssim \left(\ln\frac{1}{|u|}\right)^{1-p},\ \forall 0<|u|\ll 1\ .
\eneqa
Our first main result is the following.

\begin{theorem}[Global existence]\label{thm-sdge}
Let $n=2$ and $p>3$. Considering \eqref{nlw} with 
$F(u)=F_p(u)$ satisfying \eqref{nonlinear},
there exists $\vep_0(p)>0$ so that
the problem admits a global weak solution $u\in C(\R\times \htw)$ for any $|\vep|<\vep_0$
and initial data $(u_0, u_1)\in C^1(\htw) \times C(\htw)$
 satisfying \eqref{data}.
\end{theorem}
%\begin{remark} We call the solution $u$ of \eqref{nlw} is a weak solution, if $u\in C(\R^{+}\times\hth)$ and satisfies \eqref{ie}. \end{remark}

On the other hand, we consider the problem of blow up for relatively small powers. It turns out that $p_c(2)=3$ for
\eqref{nlw} with \eqref{eq-nlterm} or more general \eqref{nonlinearity}, which is ensured by the following blow up result.

%Furthermore, it turns out that $p_c(2)=3$ for
%\eqref{nlw} with \eqref{eq-nlterm}
%
%Together with the following blow up result, 
%
%To determine the critical power $p_c(2)$, we consider the problem of blow up for relatively small powers. It turns out that $p_c(2)=3$ for
%\eqref{nlw} with \eqref{eq-nlterm}, which is ensured by the following blow up result.
%For this purpose,

%\begin{theorem}[Formation of singularity]\label{thm-bu}
%Let $1<p<3$.
%Considering \eqref{nlw} with 
%$F(u)=F_p(u)$ given by \eqref{eq-nlterm}, $u_0=0$ and $u_1=u_1(\tau)$ which is a spherically symmetric nonnegative continuous compactly supported function satisfying 
%$$u_1(\tau)\ge1\quad\tau\in[\tau_0,3\tau_0]$$
%with some constant $\tau_0>0$,
%for arbitrary $\vep>0$,  the corresponding weak solution will blow up in finite time.
%\end{theorem}
%
%
%Actually, similar to \cite{wang2023wave}, our proof could be adapted for general nonlinearities: we assume $F$ is a $C^1(\R)$ function so that,
%$F(0)=F'(0)=0$,
%\beeq\label{nonlinearity}
%\left\{\begin{array}{l}
%F(u)\gtrsim \left(\ln\frac{1}{|u|}\right)^{1-p}|u|,\ |u|\ll 1\ ,\\
%F(u)\gtrsim |u|^{q},\ |u|\gtrsim 1\ ,
%\end{array}
%\right.
%\eneq
% for  $q>1$ and $p\in (1,3)$.
 
\begin{theorem}[Formation of singularity]\label{thm-bu}
Let $1<p<3$,
and $F$ be a $C^1(\R)$ function so that,
$F(0)=F'(0)=0$, 
\beeq\label{nonlinearity}
\left\{\begin{array}{l}
F(u)\gtrsim \left(\ln\frac{1}{|u|}\right)^{1-p}|u|,\ |u|\ll 1\ ,\\
F(u)\gtrsim |u|^{q},\ |u|\gtrsim 1\ ,
\end{array}
\right.
\eneq
 for some $q>1$.
Considering \eqref{nlw} with such
$F(u)$, %e.g. $F_p(u)$ in \eqref{eq-nlterm}, 
$u_0=0$ and $u_1=u_1(\tau)$ which is a spherically symmetric nonnegative continuous compactly supported function satisfying 
$$u_1(\tau)\ge1\quad\tau\in[\tau_0,3\tau_0]$$
with some constant $\tau_0>0$,
for arbitrary $\vep>0$,  the corresponding weak solution will blow up in finite time.
\end{theorem}
 
\begin{remark}
We want to point out that the initial data $(u_0,u_1)$ in Theorem \ref{thm-bu} can be any nontrivial $C_c^1\times C_c$ data on physical case $n=3$ in \cite{wang2023wave}, due to the strong Huyghens’ principle, which is not true in even dimensions.
\end{remark}

%At last, we would like to discuss some further problems, before concluding the introduction. 
%Concerning the problem
%\eqref{nlw} with \eqref{eq-nlterm},
%the first natural problem is to determine the critical powers $p_c(n)$ for $n\neq 3$. For this problem, heuristically, in view of the sharp linear decay of the form $(\sinh t)^{-\rho}\sim e^{-\rho t}$ ($t>1$), we expect similar asymptotic behavior 
%$u(t,x)\sim e^{-\rho t}$, along the light cone,
%for $p>p_c(n)$, from which the nonlinear problem  \eqref{nlw}  is expected to behave like 
%$$|\partial_{t}^2 u-(\De_{\hn}+\rho^2)u|=|F(u)|\les \<t\>^{-(p-1)} |u|\ .$$
%Viewing the multiplication operator $\<t\>^{-(p-1)}$ as a short range perturbation of the operator $\pt^2$, it seems natural to
%conjecture that $p_c(n)$ is precisely $3$, regardless of the spatial dimension. More precisely,
%we conjecture that there exists
%$\de=\de(n)>0$ so that 
%we have global existence, with small data, for any $p\in (3, 3+\de(n))$, while for
%$p\in (1,3)$, there exist
% some data $(u_0, u_1)$ so that there is no global solutions for any $\vep>0$.
% For the case with $1<p<p_c(n)$, besides the blow up results, it is also interesting to determining the sharp lifespan, in terms of $\vep$. Furthermore, the more challenging problem may be to understand the critical behavior when $p=p_c(n)$.

\subsection*{Organization of this paper} Our paper is organized as follows. We recall 
the fundamental dispersive estimate for the linear solution $u^0$ and 
illustrate the essential 
 formula for blowup in Section \ref{pre}. In Section \ref{global}, we prove the global existence by iteration, for any $p>3$, by exploiting the dispersive estimate.
The result for the formation of singularity, Theorem \ref{thm-bu}, is presented in Section \ref{blowup}. % for which we closely follow the idea of John \cite{MR535704}. 
Finally, 
in the appendix, we present  some fundamental estimates used in Section \ref{pre} and Section \ref{global}.
%the solution representation formula \eqref{mean}. 

\subsection*{Notation} 
\begin{itemize}
\item
We use $A\lesssim B$ to denote
$A\leq CB$ for some large constant C which may vary from line to
line and depend on various parameters, and similarly we use $A\ll B$
to denote $A\leq C^{-1} B$. We employ $A\sim B$ when $A\lesssim
B\lesssim A$.
\item $d(x,y)$ is the geodesic distance between $x,y$ in $\htw$, and if $x$ is the origin $O$, we denote $|y|=d(O,y)$.
\item $S_t(x):=\{y\in\htw,d(x,y)=t\}$ denotes the hyperbolic sphere with center $x$ and radius $t$.
\item $M^r f(x)=\frac{1}{|S_r(x)|}\int_{S_r(x)}f(y)d\sig_y$ denotes the spherical mean of $f$ over $S_r(x)$. If the center is the origin $O$, for any function $w(t,x)$ with parameter $t$, we simply denote $\wt{w}(t,r)%=\wt{w}(t,x)
:=M^r(w(t, \cdot))(O)$. %for $|x|=r$.
%\item For $(O,t_0)\in \hth\times\R$, we denote the forward and backward cones with vertex $(O,t_0)$ by
%$$\Gamma^{\pm}(O,t_0)=\{(x,t): d(O,x)\le \pm (t-t_0)\}\ .$$
%$\mathit{notation\,:\ A\lesssim_{y}B\ means\ that\ A\leq CB}$ $\mathit{where\ C\ is\ a\ constant}$ $\mathit{only}$ $\mathit{depending}$ $\mathit{on\ y}$.
\item we denote $\<t\>=(1+t^2)^{\hf}$, $x\vee y=\max(x,y)$, $x\wedge y=\min(x,y)$.
\end{itemize}

\section{Preliminary}\label{pre}

Inside the forward light cone of the Minkowski space $\La=\{(\tau,z)\in \R^{1,n}: |z|<\tau\}$, we introduce coordinates 
\beeq\label{varitrans}
s=|z|,\ \tau=e^t\cosh r,\ s=e^t\sinh r,\ r\in[0,\infty),\ t\in\R.
\eneq
Viewing $\hn$ as
the embedded spacelike hypersurface with $t=0$, we have
the natural metric $g_{\hn}=dr^2+(\sinh r)^2d\om^2$, induced from the Minkowski metric $g=-d\tau^2+dz^2=-d\tau^2+ds^2+s^2d\om^2$, where  $\om\in\ms^{n-1}$. This illustrates that $(r,\om)$ is the natural geodesic polar coordinates in $\hn$.

Considering the linear wave equations \beeq\label{wave}
%\tag{NLW}
\begin{cases}
\partial_{t}^2 u-(\De_{\hn}+\rho^2)u=F\ , x\in\hn\ ,\\
u(0,x)= u_0, u_t(0,x)= u_1\ ,
\end{cases}
\eneq
%it is well known that
Duhamel's principle tells us that
 \eqref{wave} is equivalent to the integral equation
\beeq\label{ie}
%\tag{IE}
u(t,x)= u^0(t,x)+(LF)(t,x)= u^0(t,x)+\int_0^t I(\tau, x, F(t-\tau))d\tau,
\eneq
where $u^0= \pt I(t,x, u_0)+  I(t,x,u_1)$, and $I(t,x,u_1)$ is the solution for the linear homogeneous equation with data $u_0=0$.

In the proof of Theorem \ref{thm-sdge},
a fundamental result to be used is the following a priori dispersive estimate for the linear solution $u^0$, which is available in \cite[Theorem 6]{MR1443052} with \cite[Theorem 3.1]{MR1295101}.
For completeness, we present a proof.

\begin{lem}[Linear estimates]\label{formulae}
Let $n=2$ and $k>0$,
there exists $N_k>1$ so that
we have the estimate 
\beeq\label{eq-linear}
|u^0(t,x)|\le  N_k(\cosh |x|)^{-\hf}(\cosh (|t|-|x|))^{-\hf}K_k(|t|-|x|)\ ,
\eneq
for any solutions to  \eqref{wave} with $F=0$ and the initial data $(u_0, u_1)\in C^1\times C$ satisfying \eqref{data}. And the factor $K_k$ is defined by
$$K_k(s)=\begin{cases}(\cosh s)^{-k+\hf}&0<k<\hf\\
\<s\> &\quad k=\hf\\1&\quad k>\hf\end{cases}.$$
\end{lem}
\begin{proof}
Without loss of generality, we assume $t>0$. At first, we recall that 
%have the formula of the solution
%\beeqa\label{linear}
$u^0(t,x)=\partial_tI(t, x, u_0)+I(t, x, u_1)$
%\eneqa
on $(0,\infty)\times\htw $ with
\beeqa\label{mean}
\begin{split}
I(t,x, u_1)&=\frac{2}{\omega_2}(2\cosh t-2\cosh|\cdot|)_+^{-\hf}*u_1(x)\\
&=\int_0^t\frac {\sinh s} {\sqrt{2\cosh t-2\cosh s}}M^s u_1(x)ds,
\end{split}
\eneqa
where $\frac{2}{\omega_2}(2\cosh t-2\cosh|\cdot|)_+^{-\hf}$ is the fundamental solution (see \cite[Chapter 8 (5.14)]{MR2743652}), and $\omega_n=\frac{2\pi^{(n+1)/2}}{\Gamma(\frac{n+1}{2})}$ is the volume of n-dimensional sphere. %\cite{Taylor}+径向的球面平均公式

%which could be obtained from the fundamental solution  %
%\begin{equation*}
%\begin{split}
%I(t,x,g)&=\mathscr{X}_{+}^{-\hf}(2\cosh t-2\cosh|\cdot|)*g(x)\\
%&=(2\cosh t-2\cosh|\cdot|)_+^{-\hf}*g(x)
%\end{split}
%\end{equation*}
%which could be obtained %from the  d'Alembert's formula together with 
%from a relation between the wave operators on hyperbolic space and that on Minkowski space, see Appendix for a sketch of the proof.
 
Owing to $|u_1(x)|\le \theta_k(x)=(\cosh |x|)^{-k-\hf}$, we have %the following estimate with $|x|=r$
$$|I(t,x,u_1)|\le I(t, x,\theta_k)=\int_0^t\frac {\sinh s} {\sqrt{2\cosh t-2\cosh s}}(M^s \theta_k)(x)ds\ ,$$
where
\beeq\label{radialmean}
\begin{split}
(M^t \theta_k)(x)
&=\frac{1}{\pi}\int_{|r-t|}^{r+t}\frac{\theta_k\sinh \la}{(\cosh(r+t)-\cosh\la)^{\hf}(\cosh\la-\cosh(r-t))^{\hf}} d\la\\
&=\frac{1}{\pi}\int_{|r-t|}^{r+t}\frac{(\cosh\la)^{-k-\hf}\sinh \la}{(\cosh(r+\la)-\cosh t)^{\hf}(\cosh t-\cosh(r-\la))^{\hf}} d\la\ ,
\end{split}
\eneq
with $|x|=r$. One can find the derived formula in \cite[(3.5)]{MR1295101}, 
which is the parallel formula of the spherical mean function on Euclidean spaces. %one can find the derived formula in \cite[(3.5)-(3.7)]{MR1295101}.
%see \eqref{eq-sph-ave}.
%We claim that we could prove an even better estimate for $|I(x,t,g)|$:
%%Then the proof of \eqref{eq-linear} amounts to the proof of
%\beeq\label{eq-linear-claim}
%|(M^t \theta_k)(x)|\le C_k \frac{1}{(\cosh t)(\cosh r) (\cosh (t-r))^{k}}\ .
%\eneq
%we need only to prove

%Before proving \eqref{eq-linear-claim}, let us check that it is strong enough to conclude \eqref{eq-linear}. Actually, when $u_0=0$, it is stronger than
%\eqref{eq-linear}, due to the fact that $\tanh t\in [0, 1]$.
%For the case with $u_1=0$, by \eqref{mean}, we see that
%$$u^0=
%\partial_t I(t, x, u_0)=\ ,
%$$ for which it remains to control
%$\pt (M^t u_0)(x)$.
%%%%%%%%%%这个地方即使是用M^t\theta_k去算I(t,r,\theta_k)，计算量也不小，不如直接用附录里的Lemma
By Lemma \ref{linear-tech}, we take $a(s)=2\cosh s$, $f(\la)=(\cosh\la)^{-k-\hf}\sinh \la$, and it is easy to verify $a$ satisfies conditions of Lemma \ref{partial-tech}. Then we have
\beeq\label{I}
I(t, x,\theta_k)=\frac1 \pi W(t,r,f)\le\begin{cases}
\int_{r-t}^{r+t}\frac{(\cosh\la)^{-k-\hf}\sinh \la}{(2\cosh(r+\la)-2\cosh t)^{\hf}}d\la\quad r\ge t\\
\int_0^{t+r}\frac{(\cosh\la)^{-k-\hf}\sinh \la}{|2\cosh(r+\la)-2\cosh t|^{\hf}}d\la\quad t>r
\end{cases}.
\eneq
Observe the connection between integrals of $M^t\theta_k(x)$ and \eqref{I}, we claim that
\beeq\label{eq-linear-claim}
(M^t \theta_k)(x)\le N_k(\cosh (t+r))^{-\hf}(\cosh (t-r))^{-k}\ ,
\eneq
and it is strong enough to conclude \eqref{eq-linear}. Actually, when $u_0=0$, on the one hand, it is directly true that
\beeq\label{part}
\int_{|r-t|}^{r+t}\frac{(\cosh\la)^{-k-\hf}\sinh \la}{(2\cosh(r+\la)-2\cosh t)^{\hf}}d\la\le(\cosh t)^{\hf}M^t\theta_k(x)\le N_k(\cosh r)^{-\hf}(\cosh (t-r))^{-k}.
\eneq
On the other hand, for $t>r$, without loss of generality, we consider $t-r\gg 1$. With \eqref{part} and
\beeq\label{coshd}
\cosh b\cosh c\le\cosh(b+c)\le2\cosh b\cosh c\quad (b,c>0)\ ,
\eneq
the rest of the last integral in \eqref{I}
% that gets rid of \eqref{part} in \eqref{I} 
turns to
\beqa
&&\int_0^{t-r} \frac{(\cosh\la)^{-k-\hf}\sinh \la}{(\cosh t-\cosh (r+\la))^{\hf}}d\la\\
&\le&N(\cosh t)^{-\hf}\int_0^{t-r-1}(\cosh\la)^{-k-\hf}\sinh \la d\la\\
&&+(\cosh(t-r-1))^{-k-\hf}\int_{t-r-1}^{t-r}\frac{\sinh\la}{\sinh(r+\la)}\frac{\sinh(r+\la)}{(\cosh t-\cosh (r+\la))^{\hf}}d\la\\
&\le&N_k(\cosh t)^{-\hf}K_k(t-r)+(\cosh r)^{-1}(\cosh(t-r-1))^{-k-\hf}(\cosh t)^{\hf}\\
&\le&N_k(\cosh r)^{-\hf}(\cosh(t-r))^{-\hf}K_k(t-r).
\eeqa
%(\cosh(t-r))^{-k+\hf}

For the case with $u_1=0$, by taking $a(s)=2\cosh s$, $v(t,x)=M^tu_0(x)$ in Lemma \ref{partial-tech}, we deduce
$$|\pa_tI(t, x, u_0)|\le\frac{\sinh t}{\sinh \frac t 2} (M^t|u_0|)(x)+R|\pa_t(M^tu_0)|(t,x).$$
%By Lemma \ref{partial-tech} with $a(s)=2\cosh s$, $v(t,x)=M^tu_0(x)$, 
%and 
%\beqa
%H(s)&=&\frac{2\sinh t-2\sinh s}{\sqrt{2\cosh t-2\cosh s}}\\
%H^{\pri}(s)&=&\frac{-2\cosh s\cosh t+2(\cosh s)^2-2\cosh(t-s)+2}{(2\cosh t-2\cosh s)^{3/2}}\le0
%\eeqa
%tells us $a$ satisfies conditions in 
By \eqref{eq-linear-claim}, we only need to control $\pt (M^t u_0)(x)$, for which we introduce
a Lorentz boost $\psi_x\in SO(1,2)$ such that $\psi_x(O)=x$ and preserves the metric. It is known that $\psi_x S_t(O)=S_t(x)$ %preserves the metric and,
 and for fixed $x, y$, $\psi_x (t,y)=(r_x(t,y),\om_x(t,y))=\gam_{x,y}(t)$ is a geodesic curve with
$|\pt\gam_{x,y}(t)|=1$. Then 
\beqa
\partial_t(M^t u_0)(x) &=&\frac{1}{2\pi}\int_{S_1(O)} \partial_t(u_0(\psi_x(t,y))) d\sig_y\\
&=&\frac{1}{2\pi}\int_{S_1(O)} \partial_t(u_0(r_x(t,y),\om_x(t,y))) d\sig_y\\
&=&\frac{1}{2\pi}\int_{S_1(O)} \<\nabla u_0, \pt \gam_{x,y}(t)\>_{g_{\htw}} %(\gam_{x,y}(t))
d\sig_y\ ,
%\\&\les&M^t\theta_k(x),
\eeqa
and also
$$|\partial_t(M^t u_0)(x)|
\le\frac{1}{2\pi}\int_{S_1(O)} |\nabla u_0(\gam_{x,y}(t))|d\sig_y
\le M^t( \theta_k)(x)\ ,
$$
thanks to the assumption \eqref{data}.

To conclude the proof, we prove \eqref{eq-linear-claim}. We only consider the nontrivial case $r+t\gg1$. Without loss of generality, we set $r+t>|r-t|+2$, and then we have
\beqa
&&\int_{|r-t|}^{r+t}\frac{(\cosh\la)^{-k-\hf}\sinh \la}{(\cosh(r+t)-\cosh\la)^{\hf}(\cosh\la-\cosh(r-t))^{\hf}} d\la\\
&=&\int_{|r-t|}^{|r-t|+1}+\int_{|r-t|+1}^{t+r-1}+\int_{t+r-1}^{t+r}\\
&\le&2(\cosh(r+t))^{-\hf}(\cosh(r-t))^{-k-\hf}\int_{|r-t|}^{|r-t|+1}\frac{\sinh \la}{(\cosh\la-\cosh(r-t))^{\hf}}d\la\\
&&+2(\cosh(r+t))^{-\hf}\int_{|r-t|+1}^{t+r-1}(\cosh\la)^{-k-1}\sinh \la d\la\\
&&+N_k(\cosh(r+t))^{-k-1}\int_{t+r-1}^{t+r}\frac{\sinh \la}{(\cosh(r+t)-\cosh\la)^{\hf}}d\la\\
&\le&N_k(\cosh(r+t))^{-\hf}(\cosh(r-t))^{-k}.
\eeqa

\end{proof}

	Finally, the following estimate is useful for iterating the lower bound of solutions in the proof of Theorem \ref{thm-bu}.
\begin{lem}\label{iterforbu}
Assume that $\phi\in C(\htw)$ is a nonnegative spherically symmetric function, and then for any fixed constant $\tau_0>0$, there exist a constant $0<C_0\le1$ depending only on $\tau_0$, such that %the following estimate 
\beeq\label{ldomain}
I(t,x,\phi)\ge\frac{C_0}{(\sinh r)^{\hf}}\int_{t\vee r}^{t+r}\phi(\la)(\sinh\la)^{\hf}d\la
\eneq
holds for $|x|=r>\hf\tau_0$.
Particularly,  for $|x|=r>\hf\tau_0$ and $|t-r|>\frac18\tau_0$, we have
\beeq\label{sdomain}
I(t,x,\phi)\ge\frac{C_0}{(\sinh r)^{\hf}}\int_{|t-r|}^{t+r}\phi(\la)(\sinh\la)^{\hf}d\la.
\eneq

\end{lem}
\begin{proof}
Based on \eqref{mean} and \eqref{radialmean}, we can rewrite $I(t,x,\phi)$ with $a(s)=2\cosh s$ as %follows
\begin{equation*}
\begin{split}
I(t,r,\phi)&=\frac{1}{\pi}\int_0^t\int_{|r-s|}^{r+s}\frac{\phi(\la)\sinh\la a^{\prime}(s)d\la ds}{(a(t)-a(s))^{\hf}(a(r+s)-a(\la))^{\hf}(a(\la)-a(r-s))^{\hf}}\\
&=\frac{1}{\pi}\int_0^t\int_{|r-s|}^{r+s}\frac{\phi(\la)\sinh\la a^{\prime}(s)d\la ds}{(a(t)-a(s))^{\hf}(a(r+\la)-a(s))^{\hf}(a(s)-a(r-\la))^{\hf}}\ .
\end{split}
\end{equation*}
According to the proof of Lemma \ref{linear-tech}, we obtain
\beeqa
\nonumber I(t,r,\phi)&\ge&\frac{1}{\pi}\int_{|r-t|}^{r+t}\frac{\phi(\la)\sinh\la}{(a(r+\la))^{\hf}}\int_{|r-\la|}^t\frac{a^{\prime}(s)dsd\la}{(a(t)-a(s))^{\hf}(a(s)-a(r-\la))^{\hf}}\\
\label{lowerbd}&=&\int_{|r-t|}^{r+t}\frac{\phi(\la)\sinh\la}{(2\cosh (r+\la))^{\hf}}d\la\ .
\eneqa

For  $r>\hf\tau_0$, it is easy to conclude \eqref{ldomain} from \eqref{lowerbd}
%$$I(t,x,\phi)\ge\frac{C}{(\sinh r)^{\hf}}\int_{t\vee r}^{t+r}\phi(\la)(\sinh\la)^{\hf}d\la\ ,$$
with $(\tanh\la\tanh r)^{\hf}\ge2C_0$ for $r,t\vee r>\frac18\tau_0$.

Particularly, for $r>\hf\tau_0$ and $|t-r|>\frac18\tau_0$, we have \eqref{sdomain} directly from \eqref{lowerbd} with $(\tanh\la\tanh r)^{\hf}\ge2C_0$ for $r,|t-r|>\frac18\tau_0$.
%\beqa
%I(t,r,\phi)\ge\frac{C}{(\sinh r)^{\hf}}\int_{|t-r|}^{t+r}\phi(\la)(\sinh\la)^{\hf}d\la\ .
%\eeqa
%with $(\tanh\la\tanh r)^{\hf}\ge2C$ for $r,|t-r|>\frac14\tau_0$. 

\end{proof}

%\beeqa\label{linear}
%u^0(t,x)=\partial_tI(x,t,\vep u_0)+I(x,t,\vep u_1)
%\eneqa
%on $\hth\times[0,\infty)$ with
%%where $M^tf(x)$ is the spherical mean of $f$ over $S_t(x)$,
%\beeqa\label{mean}
%I(x,t,f)=\sinh t\cdot M^tf(x),
%\eneqa
%and if $g(t,x)$ is a radial continuous function in $\hth$ for any fixed $t$, we can obain the formula ($|x|=r$)
%\beeq\label{main}
%Lg(r,t)=\int_0^tI(x,s,g(t-s))ds=\frac{1}{2\sinh r}\int_0^t\int_{|t-s-r|}^{t-s+r}g(\la,s)\sinh\la d\la ds.
%\eneq
%The fundamental formula \eqref{main} follows from simple calculations in the appendix.

\section{Global existence}\label{global} %Proof of Theorem \ref{thm-sdge}}%\label{sec:prf}

In this section, we give the proof of Theorem \ref{thm-sdge}, for which we rewrite \eqref{nlw} into the following integral equation
\beeq\label{ie2}
%\tag{IE}
u(t,x)= \vep u^0(t,x)+(LF_p(u))(t,x)= \vep u^0(t,x)+\int_0^t I(t-\tau, x, F_p(u(\tau, \cdot))d\tau,
\eneq
where $u^0= \pt I(t,x, u_0)+  I(t,x,u_1)$ is the homogeneous solution with data $(u_0, u_1)$.
By Lemma \ref{formulae}, we have
\beqa
|u^0(t,x)|\le   N_k(\cosh |x|)^{-\hf}(\cosh (t-|x|))^{-\hf}K_k(t-|x|)\ ,
\eeqa
for any
$u_0\in C^1$, $u_1\in C$ satisfying \eqref{data}.
Then for any $h>0$, there exists a constant $N_h>N_k$ such that
\beeq\label{linear}
|u^0(t,x)|\le  %\frac{ N_k K_k(t-r)}{(\cosh r)^{\hf}(\cosh (t-r))^{\hf}}\le  
\frac{N_h}{e^{r/2}\<t-r\>^{h}}
:=\frac{  N_h}{\Phi_h(t,r)}\ ,
\eneq
with $|x|=r$ and the fact of 
\beqa
\frac 12  e^{|t|}\le \cosh t\le e^{|t|},
%|\tanh t|\le 1,
\forall t\in\R\ .
\eeqa

%Based on the elementary inequality
%%, 
%\beqa
%\frac 12  e^{|t|}\le \cosh t\le e^{|t|},
%%|\tanh t|\le 1,
%\forall t\in\R\ ,
%\eeqa
%we observe that, 
%where we denote $\ol{f}(t,r)=\max_{|x|=r}|f(t,x)|$.
% and $\Phi(t,r)=e^{r}\<t-r\>^{h}$ with undetermined $h$.%, $\Phi_0^{-1}=2\vep N_h \Phi^{-1}$ for simplicity.

For fixed $h>0$ to be specified later,
the global existence of the solution $u$ of \eqref{ie2} will be proved by iteration, for which %we have to introduce a suitable norm. 
we define the (complete) solution space and the 
solution map
 $T u$ as follows
$$X_{\vep} = \{u\in C([0,\infty)\times \htw): \|u\|:=\|\Phi_hu\|_{L^\infty_{t,x}}\le 2\vep N_h\}\ ,$$
$$(T u)(t,x)=\vep u^0+LF_p(u)\ .$$
Then the proof is reduced to the following key nonlinear estimates
in light of Banach's contraction principle, and 
the process has been detailed 
%the reader can find the detailed illustration 
in our previous paper \cite{wang2023wave}.

\begin{lem}[Nonlinear estimates]\label{mainlem}
Let $p>3$, $h\in(1,p-2)$ and
$F_p$ be the $C^1$ function satisfying \eqref{nonlinear}. There exists $\vep_0>0$ so that
for any $\vep\in (0, \vep_0]$,
we have
\beeqa\label{compressed}
\|LF_p(u)-LF_p(v)\|\le \frac 12 \|u-v\|, \forall u, v\in X_\vep\ .
\eneqa
\end{lem}
%\le\|L(|F^{\prime}(\Phi^{-1}_0)|\Phi^{-1}_0)\|
%Actually, with the help of Lemma \ref{mainlem}, we 
%know that for $u\in X_{\vep_0}$,
%$$\|LF_p(u)\|=\|LF_p(u)-LF_p(0)\|\le \vep_0 N_h\ ,$$
%which tells us that
%$$\|Tu\|\le \|\vep u^0\|+\|LF_p(u)\|\le 2\vep_0 N_h\ ,$$
%i.e., $Tu\in X_{\vep_0}$.
%In addition,
%by \eqref{compressed},
%we have
%$$\|Tu-Tv\|=\|LF_p(u)-LF_p(v)\|\le \frac 12 \|u-v\|\ ,$$
%which ensures that
%$T: X_{\vep_0}\to X_{\vep_0}$ is a contraction map, and the fixed point is the desired solution.
\subsection{Proof of Lemma \ref{mainlem}}
%\begin{proof}[Proof of Lemma \ref{mainlem}]
%\textbf{Proof of Lemma \ref{mainlem}.}
By \eqref{nonlinear}, there exists $A>2N_h$ so that
\beeqa\label{nonlinear2}
|F_p'(u)|\le A \left(\ln \frac{1}{|u|}\right)^{1-p}\triangleq G(u), \forall |u|\le \frac 1 A\ .
\eneqa

%At first,
For any $u\in X_\vep$, as $\Phi_h\ge 1$, we know that
$$\|u\|_{L^\infty_{t,x}}\le
\|\Phi_h u\|_{L^\infty_{t,x}}=
 \|u\|\le 2 \vep N_h\le \frac 1 A\ ,$$
provided that $\vep\in (0, 1/(2N_h A)]=(0, \vep_1]$, for which we assume in what follows.
Then for any $u, v\in X_\vep$,
in view of
\eqref{nonlinear2} and the monotonicity of $G$,  we get
$$|F_p(u)-F_p(v)|\le G(\max(|u|,|v|)) |u-v|\le
%(G(u)+G(v)) |u-v|\le 
G\left(\frac{2\vep N_h}{\Phi_h}\right)  |u-v|
\ .
$$
Recall $\Phi_h(t,r)=e^{r/2}\<t-r\>^{h}\ge e^{r/2}$, we see that
\beqa
|F_p(u)-F_p(v)(t,x)|&\le&A\left(\ln\frac{1}{2\vep N_h}+r/2\right)^{1-p}  |u-v| \\
&\le&A2^{p-1}\left(\ln\frac{1}{\vep }+r\right)^{1-p} \frac{\|u-v\|}{\<t-r\>^h(\cosh r)^{\hf} }\ .
\eeqa

%By \eqref{ie2}, \eqref{mean} and \eqref{eq-sph-ave},  as well as the fact
%$$|\ol{u}(t,r)|\le \frac{\|u\|}{\Phi_h (t,r)}
%\le\frac{\|u\|}{\<t-r\>^h\cosh r }
%\ ,$$

By \eqref{ie2}, \eqref{mean}, \eqref{radialmean} and \eqref{I}, with $a(s)=2\cosh s$, $f_{\tau}(\la)=\left(\ln\frac{1}{\vep }+\la\right)^{1-p} \frac{\sinh \la}{\<\tau-\la\>^{h}(\cosh \la )^{\hf}}$ we have %the estimate
\beqa
&&|(LF_p(u)-LF_p(u))(t,x)|\\
%&\le&L \left(G\left(\frac{2\vep N_h}{\Phi_h}\right)  \ol{u-v} \right)(t,r)\\
&\le&\frac{A2^{p-1}\|u-v\|}{\pi}\int_0^tW(t-\tau,r,f_{\tau})d\tau
%&=&\frac{1}{2\sinh r}\int_0^t\int_{|t-s-r|}^{t-s+r}\left(\ln\frac{1}{2\vep N_h}+\la\right)^{1-p}  \ol{u-v}(s, \la) \sinh \la  d\la ds\\
%&\le&\|u-v\|\int_0^t\int_{|t-s-r|}^{t-s+r}\left(\ln\frac{1}{2\vep N_h}+\la\right)^{1-p}\frac{\tanh \la}{\<s- \la\>^h} d\la ds \\
%&\le&\frac{ \tanh (t+r)}{2\sinh r}\|u-v\| \int_{|t-r|}^{t+r}\int_{-\be}^{t-r}\left(\ln\frac{1}{2\vep N_h}+\frac{\be-\al}{2}\right)^{1-p}\<\al\>^{-h}d\al d\be
%&:=& J(t,r)\|u-v\| 
\ .
\eeqa
With the help of the above estimate, the proof of
 \eqref{compressed} is then reduced to the proof of
  the following claim. Considering the main integral
\beeq\label{claim}
\int_0^tW(t-\tau,r,f_{\tau})d\tau\ ,%\ll (\cosh r)^{-\hf}\<t-r\>^{-h} ,
\eneq
%\int_{|t-r|}^{t+r}\int_{-\be}^{t-r}\left(\ln\frac{1}{2\vep N_h}+\frac{\be-\al}{2}\right)^{1-p}\<\al\>^{-h}d\al d\be \ll\frac{\tanh r}{\tanh (t+r)}\<t-r\>^{-h}\ ,
we claim that
$$\eqref{claim}\ll (\cosh r)^{-\hf}\<t-r\>^{-h},$$
provided that
$p>3$, $h\in(1,p-2)$ and
 $\vep$ is sufficiently small.

Based on Lemma \ref{linear-tech}, we will divide the proof into two separate cases: $r>t$ and $t\ge r$. We also consider the nontrivial case $r+t\gg1$.

%where we have introduced new variables of integration $\al=s-\la$, $\be=s+\la$.

\subsubsection{Case 1: $r>t$}
In this case, we have $r>t\ge t-\tau$ and then %there is the simple expression of $W$ in \eqref{claim}:
$$\eqref{claim}\le\frac{\pi}{\sqrt{2}}\int_{0}^{t}\int_{r-(t-\tau)}^{r+t-\tau}\frac{\left(\ln\frac{1}{\vep }+\la\right)^{1-p}\<\tau-\la\>^{-h}\sinh\la}{(\cosh\la)^{\hf}(\cosh(r+\la)-\cosh(t-\tau))^{\hf}}d\la d\tau.$$
Based on \eqref{coshd} and the fact of $r+\la-(t-\tau)\ge\la$, we have %the fundamental property
\beeq\label{funda}
\cosh(r+\la)-\cosh(t-\tau)\ge\cosh(r+\la)(1-\frac{1}{\cosh\la})\gtrsim\cosh(r+\la)(\la^2\wedge1).
\eneq
Therefore with $h\in(1,p-2)$ and $p>3$, \eqref{claim} can be controlled by 
\beqa
&&\int_{0}^{t}\int_{r-(t-\tau)}^{r+t-\tau}\frac{\left(\ln\frac{1}{\vep }+\la\right)^{1-p}\<\tau-\la\>^{-h}\sinh\la}{(\cosh\la)^{\hf}(\cosh(r+\la))^{\hf}(\la\wedge1)}d\la d\tau\\
&\le&(\cosh r)^{-\hf}\int_{0}^{t}\int_{r-(t-\tau)}^{r+t-\tau}\left(\ln\frac{1}{\vep }+\la\right)^{1-p}\<\tau-\la\>^{-h}d\la d\tau\\
&=&(\cosh r)^{-\hf}\int_{r-t}^{r+t}\left(\ln\frac{1}{\vep }+\la\right)^{1-p}\int_{0}^{t-|r-\la|}\<\tau-\la\>^{-h}d\tau d\la\\
&\les&(\cosh r)^{-\hf}\int_{r-t}^{r+t}\left(\ln\frac{1}{\vep }+\la\right)^{1-p}d\la\\
&\les&(\cosh r)^{-\hf}\left(\ln\frac{1}{\vep }+r-t\right)^{2-p}\ll(\cosh r)^{-\hf}\<t-r\>^{-h}.
\eeqa

\subsubsection{Case 2: ${t\ge r}$}
In this case, we have to divide the integral of \eqref{claim} into two parts $\tau\ge t-r$ and $\tau<t-r$:
\begin{align*}
\eqref{claim}\le&\frac{\pi}{\sqrt{2}}\int_{0}^{t-r}\int_{0}^{t-\tau+r}\frac{\left(\ln\frac{1}{\vep }+\la\right)^{1-p}\<\tau-\la\>^{-h}\sinh\la}{(\cosh\la)^{\hf}|\cosh(r+\la)-\cosh(t-\tau)|^{\hf}}d\la d\tau\\
&+\frac{\pi}{\sqrt{2}}\int_{t-r}^t\int_{r-(t-\tau)}^{r+t-\tau}\frac{\left(\ln\frac{1}{\vep }+\la\right)^{1-p}\<\tau-\la\>^{-h}\sinh\la}{(\cosh\la)^{\hf}(\cosh(r+\la)-\cosh(t-\tau))^{\hf}}d\la d\tau\\
=&\frac{\pi}{\sqrt{2}}\int_{0}^{t-r}\int_{0}^{(t-\tau)-r}\frac{\left(\ln\frac{1}{\vep }+\la\right)^{1-p}\<\tau-\la\>^{-h}\sinh\la}{(\cosh\la)^{\hf}(\cosh(t-\tau)-\cosh(r+\la))^{\hf}}d\la d\tau\\
&+\frac{\pi}{\sqrt{2}}\int_{0}^t\int_{|r-(t-\tau)|}^{r+t-\tau}\frac{\left(\ln\frac{1}{\vep }+\la\right)^{1-p}\<\tau-\la\>^{-h}\sinh\la}{(\cosh\la)^{\hf}(\cosh(r+\la)-\cosh(t-\tau))^{\hf}}d\la d\tau\\
=&\frac{\pi}{2\sqrt{2}}\int_0^{t-r}\int_{-\be}^{\be}\frac{\left(\ln\frac{1}{\vep }+(\be-\al)/2\right)^{1-p}\<\al\>^{-h}\sinh\la}{(\cosh\la)^{\hf}(\cosh(t-\tau)-\cosh(r+\la))^{\hf}}d\al d\be\\
&+\frac{\pi}{2\sqrt{2}}\int_{t-r}^{t+r}\int_{-\be}^{t-r}\frac{\left(\ln\frac{1}{\vep }+(\be-\al)/2\right)^{1-p}\<\al\>^{-h}\sinh\la}{(\cosh\la)^{\hf}(\cosh(r+\la)-\cosh(t-\tau))^{\hf}}d\al d\be
\end{align*}
where we introduce new variables of integration $\al=\tau-\la,\be=\tau+\la$. 
To ensure integrability, we make use of %the following property
% with the help of mean value theorem
%we have to use 
%we have another fundamental property 
%by mean value theorem
\beeq\label{fundat}
\begin{split}
&\cosh(t-\tau)-\cosh(r+\la)\gtrsim\begin{cases}\sinh(r+\la)(t-r-\be)\ &t-r-\be\le1\\\cosh(t-\tau)\ &t-r-\be\ge1\end{cases}\\
&\cosh(r+\la)-\cosh(t-\tau)\gtrsim\begin{cases}\sinh(t-\tau)(\be-(t-r))\ &\be-(t-r)\le1\\\cosh(r+\la)\ &\be-(t-r)\ge1\end{cases}
\end{split}
\eneq
instead of \eqref{funda}.
Without loss of generality, we only consider $r\gg1$ and $t-r\gg1$. Based on \eqref{fundat}, and the fact that
$$\cosh(t-\tau)=\cosh(t-r-\be+r+\la)\ge\cosh(t-r-\be)\cosh r\cosh\la$$
for $t-r-\be>0$, 
$$\sinh(t-\tau)=\sinh(t-r+1-\be+r-1+\la)\ge\cosh(t-r+1-\be)\cosh(r-1)\sinh\la$$
for $\be\le t-r+1$,
the case is reduced to following main integrals:
\beeqa
\label{fir}&&(\cosh r)^{-\hf}\int_0^{t-r-1}\int_{-\be}^{\be}\frac{\left(\ln\frac{1}{\vep }+\be-\al\right)^{1-p}\<\al\>^{-h}}{(\cosh(t-r-\be))^{\hf}}d\al d\be,\\
\label{sec}&&(\cosh r)^{-\hf}\int_{t-r-1}^{t-r}\int_{-\be}^{\be}\frac{\left(\ln\frac{1}{\vep }+\be-\al\right)^{1-p}\<\al\>^{-h}}{(t-r-\be)^{\hf}}d\al d\be,\\
\label{thi}&&(\cosh r)^{-\hf}\int_{t-r}^{t-r+1}\int_{-\be}^{t-r}\frac{\left(\ln\frac{1}{\vep }+\be-\al\right)^{1-p}\<\al\>^{-h}}{(\be-(t-r))^{\hf}}d\al d\be,\\
\label{fou}&&(\cosh r)^{-\hf}\int_{t-r+1}^{t+r}\int_{-\be}^{t-r}\left(\ln\frac{1}{\vep }+\be-\al\right)^{1-p}\<\al\>^{-h}d\al d\be.
\eneqa

Firstly, with $h\in(1,p-1)$ and sufficiently small $\vep$,
 we  focus on \eqref{fir} and divide it into
% with $h\in(1,p-1)$ and sufficiently small $\vep$
%=\int_{0}^{(t-r)/2}\int_{-\be}^{\be}+\int_{(t-r)/2}^{t-r-1}\int_{-\be}^{\be}
\beqa
&&\int_0^{t-r-1}\int_{-\be}^{\be}\frac{\left(\ln\frac{1}{\vep }+\be-\al\right)^{1-p}\<\al\>^{-h}}{(\cosh(t-r-\be))^{\hf}}d\al d\be\\
&=&\int_{0}^{(t-r)/2}(\int_{-\be}^{\be/2}+\int_{\be/2}^{\be})+\int_{(t-r)/2}^{t-r-1}(\int_{-\be}^{\be/2}+\int_{\be/2}^{\be})=:J_{ori}+J_{lar}\ .
\eeqa
There has enough decay of $t-r$ in $J_{ori}$ due to $(\cosh(t-r-\be))^{-\hf}$ as follows
\begin{equation*}
\begin{split}
%&&\int_{0}^{(t-r)/2}(\int_{-\be}^{\be/2}+\int_{\be/2}^{\be})\\
J_{ori}\les&(\cosh(t-r))^{-\frac14}\int_0^{(t-r)/2}\left(\ln\frac{1}{\vep }+\be\right)^{1-p}\int_{-\be}^{\be/2}\<\al\>^{-h}d\al d\be\\
&+(\cosh(t-r))^{-\frac14}\int_0^{(t-r)/2}\<\be\>^{-h}\int_{\be/2}^{\be}\left(\ln\frac{1}{\vep }+\be-\al\right)^{1-p}d\al d\be\\
\les&\left(\ln\frac{1}{\vep }\right)^{2-p}(\cosh(t-r))^{-\frac14}\ll\<t-r\>^{-h}\ .
\end{split}
\end{equation*}
%with $h\in(1,p-1)$ and sufficiently small $\vep$.
Meanwhile for $J_{lar}$, the decay of $t-r$ arises from
the other two terms depending on 
the division of the region of $\al$ and then
\begin{equation*}
\begin{split}
%&&\int_{(t-r)/2}^{t-r-1}(\int_{-\be}^{\be/2}+\int_{\be/2}^{\be})\\
J_{lar}\les&\left(\ln\frac{1}{\vep }+t-r\right)^{1-p}\int_{(t-r)/2}^{t-r-1}(\cosh(t-r-\be))^{-\hf}\int_{-\be}^{\be/2}\<\al\>^{-h}d\al d\be\\
&+\<t-r\>^{-h}\int_{(t-r)/2}^{t-r-1}(\cosh(t-r-\be))^{-\hf}\int_{\be/2}^{\be}\left(\ln\frac{1}{\vep }+\be-\al\right)^{1-p}d\al d\be\\
\les&\left(\ln\frac{1}{\vep }+t-r\right)^{1-p}+\left(\ln\frac{1}{\vep }\right)^{2-p}\<t-r\>^{-h}\ll\<t-r\>^{-h}\ .
\end{split}
\end{equation*}
%due to $h\in(1,p-1)$, 

%&&+\left(\ln\frac{1}{\vep }+t-r\right)^{1-p}\int_{(t-r)/2}^{t-r-1}(\cosh(t-r-\be))^{-\hf}\int_{-\be}^{\be/2}\<\al\>^{-h}d\al d\be\\
%&&+\<t-r\>^{-h}\int_{(t-r)/2}^{t-r-1}(\cosh(t-r-\be))^{-\hf}\int_{\be/2}^{\be}\left(\ln\frac{1}{\vep }+\be-\al\right)^{1-p}d\al d\be\\
%&\les&\left(\ln\frac{1}{\vep }\right)^{2-p}\<t-r\>^{-h}+\left(\ln\frac{1}{\vep }+t-r\right)^{1-p}\ll\<t-r\>^{-h}.

Similar to $J_{lar}$, with $h\in(1,p-2)$ and sufficiently small $\vep$,
\eqref{fou} is reduced to
%we deal with \eqref{fou} in the similar way to $J_{lar}$.%by dividing the region of $\al$ with $h\in(1,p-2)$.
\beqa
&&\int_{t-r+1}^{t+r}\int_{-\be}^{t-r}\left(\ln\frac{1}{\vep }+\be-\al\right)^{1-p}\<\al\>^{-h}d\al d\be
\\
&\les&\int_{t-r+1}^{t+r}\left(\ln\frac{1}{\vep }+\be\right)^{1-p}\int_{-\be}^{(t-r)/2}\<\al\>^{-h}d\al d\be\\
&&+\<t-r\>^{-h}\int_{t-r+1}^{t+r}\int_{(t-r)/2}^{t-r}\left(\ln\frac{1}{\vep }+\be-\al\right)^{1-p}d\al d\be\\
&\les&\left(\ln\frac{1}{\vep }+t-r\right)^{2-p}+\left(\ln\frac{1}{\vep }\right)^{3-p}\<t-r\>^{-h}\ll\<t-r\>^{-h}\ .
\eeqa

Finally, since \eqref{thi} is almost the same as \eqref{sec}, we only consider \eqref{sec} with $h\in(1,p-1)$ in the similar way to $J_{lar}$:
\beqa
&&\int_{t-r-1}^{t-r}\int_{-\be}^{\be}\frac{\left(\ln\frac{1}{\vep }+\be-\al\right)^{1-p}\<\al\>^{-h}}{(t-r-\be)^{\hf}}d\al d\be\\
%&=&\int_{t-r-1}^{t-r}(\int_{-\be}^{\be/2}+\int_{\be/2}^{\be})\\
&\les&\left(\ln\frac{1}{\vep }+t-r\right)^{1-p}\int_{t-r-1}^{t-r}(t-r-\be)^{-\hf}\int_{-\be}^{\be/2}\<\al\>^{-h}d\al d\be\\
&&+\<t-r\>^{-h}\int_{t-r-1}^{t-r}(t-r-\be)^{-\hf}\int_{\be/2}^{\be}\left(\ln\frac{1}{\vep }+\be-\al\right)^{1-p}d\al d\be\\
&\les&\left(\ln\frac{1}{\vep }+t-r\right)^{1-p}+\left(\ln\frac{1}{\vep }\right)^{2-p}\<t-r\>^{-h}\ll\<t-r\>^{-h}.
\eeqa

\section{Formation of singularity}\label{blowup}
In this section, we present the proof of Theorem \ref{thm-bu} with
$F(u)=F_p(u)$ satisfying \eqref{nonlinearity} and the suitable data.
%Since we will show blow up for some nontrivial data satisfying , we could set $\vep=1$ without loss of any generality.
 As an initial step, we give the local existence and uniqueness for general data.
%Proof of Theorem \ref{thm-bu}}%\label{sec:Bessel}
\subsection{Local existence and uniqueness}\label{sec-sub-bu}
We give a sketch of the proof for $t\in [0,T]$ with certain sufficiently small $T\in (0,1]$.

% The first aim is to estimate $LF(u)$. 
% For this purpose, we will introduce a new norm.
 
  Assume that $(u_0,u_1)\in C_c^1\times C_c$
 and  have their support in $|x|\le r_0$, by \eqref{linear}, we have 
\beeqa\label{lin}
|u^0(t,x)|\le
Ne^{-r/2}\chi_{r\le t+r_0}\le N \chi_{r\le t+r_0}
,
\eneqa for some $N>0$,
where $\chi$ is the characteristic function. Based on \eqref{lin}, we introduce 
%an alternative norm
%\beeqa\label{bunorm}
%\|u\|=
%\|  u(t,x)\|_{L^\infty([0,T]\times \htw)}
%%\sup_{0\le t\le T,0\le r\le t+r_0}e^r\ol{u}(t,r),
%\eneqa
%and 
the complete metric space $$X_T=\{u\in C([0,T]\times \htw): \|u\|\triangleq\|u\|_{L^\infty([0,T]\times \htw)}\le 2  N, \mathrm{supp}\ u(t)\subset \{r\le t+r_0\} \}\
 .$$
 As $F\in C^1$ with $F(0)=0=F^{\prime}(0)$, there exists $M>0$ such that
 $|F(u)-F(v)|\le M|u-v|\le M\|u-v\|$ 
 for any $u, v\in X_T$.
% Thus, for any such $u, v\in X_T$,
 and then
 it follows that
\beqa
|LF(u)(t,x)-LF(v)(t,x)|&\le &\int_0^tI(\tau,r,|F(u)-F(v)|(t-\tau))d\tau\\
&\le&M\|u-v\|\int_0^t\int_0^{\tau}\frac{\sinh s}{\sqrt{2\cosh\tau-2\cosh s}}ds d\tau\\
&=&M\|u-v\|\int_0^t\sqrt{2\cosh\tau-2}d\tau\ .
\eeqa
If $\sqrt{2}eMT\le 1/2$,
we see that
$$ 
\|LF(u)-LF(v)\|\le
\sqrt{2}eMT\|u-v\|\le \frac{1}{2}\|u-v\|,\ \forall u,v\in X_T\ .$$
%which ensures that 
%$$\|LF(u)-LF(v)\|\le  \frac{1}{2}\|u-v\|, \forall u,v\in X_T\ .$$
Then, it is clear that the map
$$(T u)(t,x)= u^0+LF(u)\ $$
is a contraction map on $X_T$, which ensures local existence and uniqueness.
%where constant $\wt{N}>N$. 

\subsection{Blow-up of the solution}
Because of the assumption of the initial data in Theorem \ref{thm-bu}, by \eqref{ie2} and Lemma \ref{iterforbu}, we have
\beeq\label{iterfir}
u(t,r)\ge\vep I(t,r,u_1)\ge\frac{C_0\vep}{(\sinh r)^{\hf}}\int_{t-r}^{t+r}u_1(\la)(\sinh\la)^{\hf}d\la\ge\frac{c_0\vep}{(\sinh r)^{\hf}},
\eneq
for some constant $0<c_0\ll1$, which is uniform in $(r,t)\in S$ with
$$S:=\{(\la,\tau):\tau_0<\tau-\la<2\tau_0,\tau+\la>3\tau_0\}\ .$$

Recalling  \eqref{nonlinearity}, there exists $0<\de_0<1$ so that
\beeq\label{eq-nonlinearity2}
F(u)\ge \de_0\left(\ln\frac{1}{|u|}\right)^{1-p}|u|, \forall |u|<\de_0\ ;\ 
F(u)\ge \de_0|u|^{q}, \forall |u|>1/\de_0\  .\eneq
Without loss of generality, we could assume $c_0\vep<\de_0(\sinh (\tau_0/2))^{\hf}$ so that, in view of \eqref{iterfir},
\beeq\label{eq-ite1}
F(u)(t,r)\ge\frac{ \de_0c_0\vep}{(\sinh r)^{\hf}}\left(\ln\frac{(\sinh r)^{\hf}}{c_0\vep}\right)^{1-p},\ 
\forall (r,t)\in S
\ .\eneq

\subsubsection{Improved lower bound}
To improve the lower bound, we introduce  the following regions for the original integral of $LF(u)$ by using Lemma \ref{iterforbu} and
 the $l$-th iteration%, with $\tau=t_2+2\de$,
\begin{equation*}
\begin{split}
R_{r,t}:&=\{(\la,\tau):\tau\ge0,\tau-\la\le t-r,t\le\tau+\la\le t+r,|t-r-\tau|\ge\frac18\tau_0\}\ ,\\
\Si_l&=\{(\la,\tau):\tau-\la>6l\tau_0,\la>\hf\tau_0\}\ ,\\
T^l_{r,t}&=\{(\la,\tau)\in R_{r,t}\cap\Si_l:\frac{6l\tau_0+t-r}{2}\le\tau-\la\le t-r\}\ .
\end{split}
\end{equation*}
\begin{figure}[H]
\begin{tabular}{cc}
\begin{minipage}[t]{0.5\linewidth}
\centering
\includegraphics[width=0.8\textwidth]{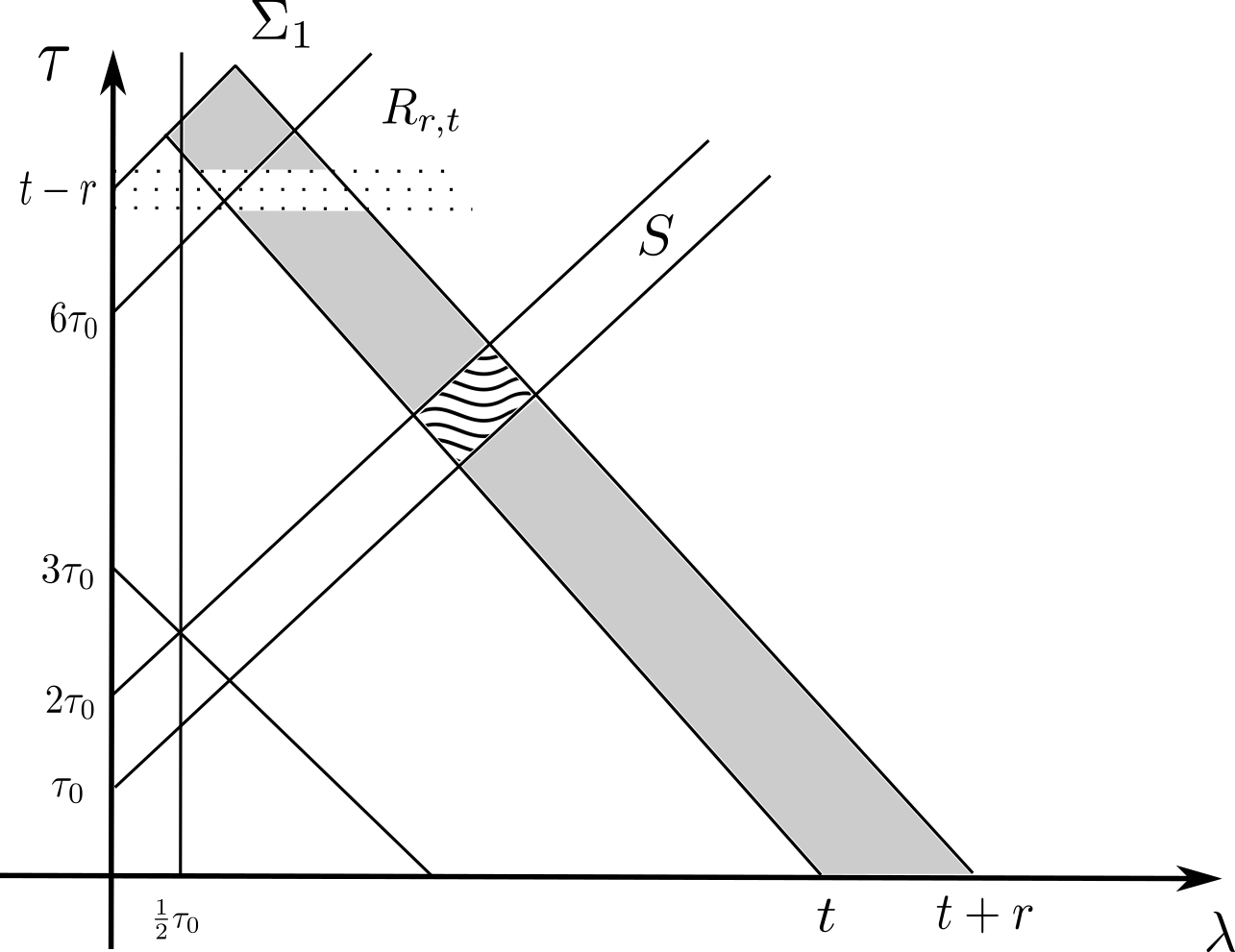}
\centerline{Figure 1}
\end{minipage}

\begin{minipage}[t]{0.5\linewidth}
\centering
\includegraphics[width=0.8\textwidth]{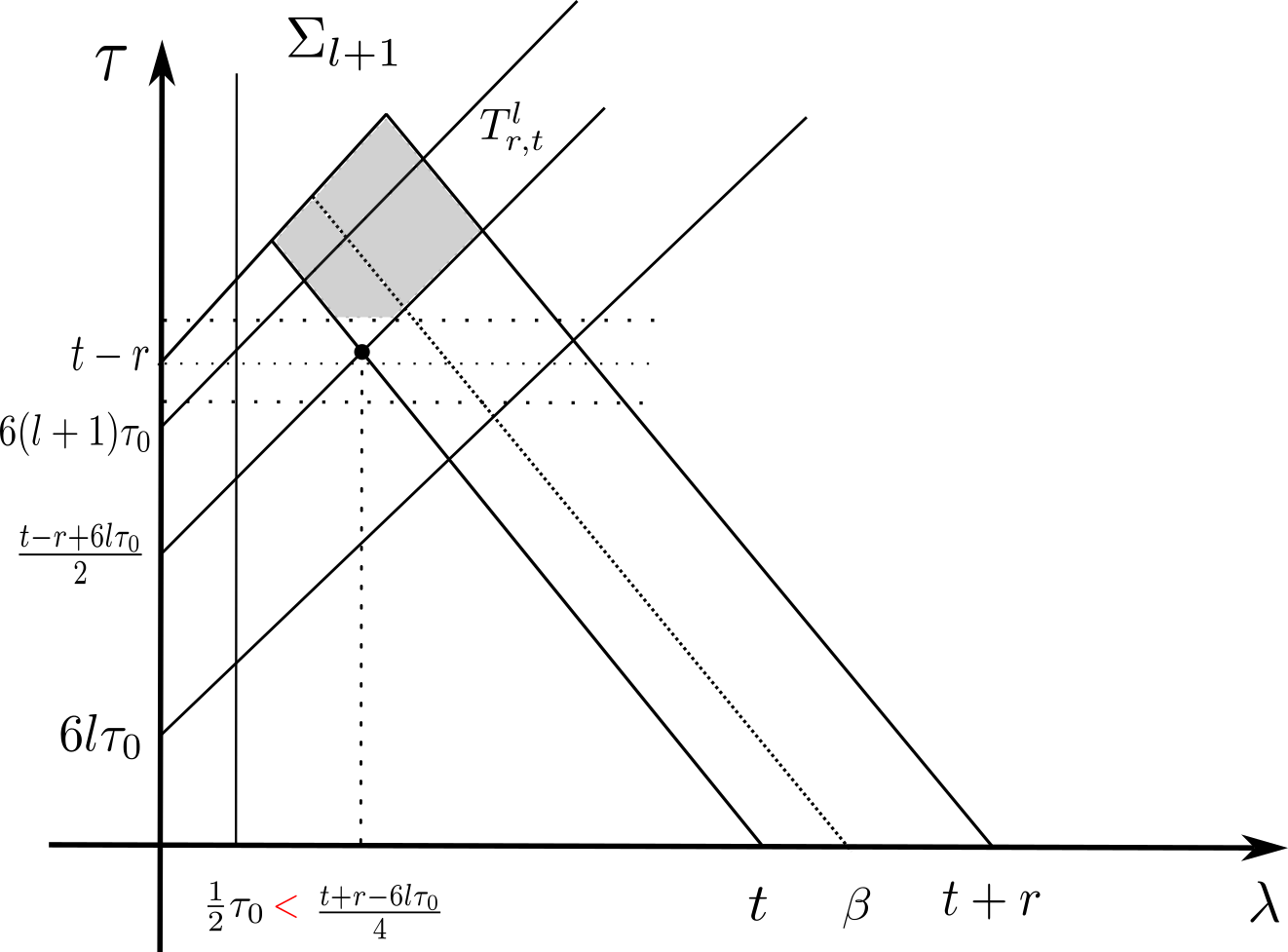}
\centerline{Figure 2}
\end{minipage}
\end{tabular}
\end{figure}
%\label{forth} \label{third}
%See Figure 1,2 for an illustration.
Based on \eqref{eq-ite1}, for any $(r,t)\in\Si_1$, we could iterate once more to obtain
\beqa
u(t,r)&\ge&\int_0^tI(t-\tau,r,F(u)(\tau))d\tau\\
&\ge&\frac{C_0}{(\sinh r)^{\hf}}\int_{[0,t]\cap|t-\tau-r|>\frac18\tau_0}\int_{|(t-\tau)-r|}^{t-\tau+r}F(u)(\tau,\la)(\sinh\la)^{\hf}d\la d\tau\\
&\ge&\frac{C_0\de_0c_0\vep}{(\sinh r)^{\hf}}\int_{R_{r,t}\cap S}\left(\ln\frac{1}{c_0\vep}+\la\right)^{1-p}d\la d\tau\\
&\ge&\frac{r\tau_0C_0\de_0c_0\vep }{4(\sinh r)^{\hf}}\left(\ln\left(\frac{1}{c_0\vep}\right)+t+r\right)^{1-p},
\eeqa
which means,  exists $c_1\in (0, c_0]$ such that
\beqa
u(t,r)\ge c_1\vep\frac{r}{(\sinh r)^{\hf}}\left(\ln\left(\frac{1}{c_1\vep}\right)+t+r\right)^{1-p}.
\eeqa

Suppose more generally that we have established an inequality of the form, 
%for the $l$-th iteration,
\beeqa\label{form}
u(t,r)\ge c_l\vep\frac{r}{(\sinh r)^{\hf}}\left(t+r+\ln\left(\frac{1}{c_l\vep}\right)\right)^{-b_l}(t-r)^{a_l},\  \forall\ (r,t)\in\Si_l,
\eneqa
for some   $b_l>0$, $c_l>0$ and
$a_l\in [ 0, b_l]$. Obviously, as $a_l-b_l\le0$, we could possibly take sufficiently small $c_l$ such that the lower bound is less than $\de_0$ and we could use the logarithmic term to iterate.
Based on \eqref{form} and Figure 2,  a further iteration for any $(r,t)\in \Si_{l+1}$ yields
\beeqa
\nonumber u(t,r)&\ge&\frac{C_0}{(\sinh r)^{\hf}}\int_{T^l_{r,t}}F(u)(\tau,\la)(\sinh \la)^{\hf} d\la d\tau\\
\nonumber&\ge&\frac{C_l\vep}{(\sinh r)^{\hf}}\int_{T^l_{r,t}}\left(\ln\left(\frac{1}{c_l\vep}\right)+\tau+\la\right)^{1-p-b_l}(\tau-\la)^{a_l}\la d\la d\tau\\
\label{domain}&\ge&\frac{C_l\vep}{(\sinh r)^{\hf}}\left(\ln\left(\frac{1}{c_l\vep}\right)+t+r\right)^{1-p-b_l}(t-r)^{a_l}\int_{T^l_{r,t}}\la d\la d\tau
%&\ge&\frac{C_l}{(\sinh r)^{\hf}}\left(\ln\left(\frac{1}{c_l\vep}\right)+t+r\right)^{1-p-b_l}(t-r)^{a_l}\int_{T_l}\be-\al d\al d\be\\
%&\gtrsim&\frac{r}{\sinh r}(t-r-l\tau)^{a_l+2}\left(\ln\left(\frac{1}{c_l}\right)+t+r\right)^{-b_l+1-p}\ ,
\eneqa
where $C_l$ are constants depending only on $l$ and different from line to line. Owing to our new coordinates, we can simplify the calculation of the last integral
\beqa
&&\int_{T^l_{r,t}}\la d\la d\tau\\
&=&\frac14\int_{t}^{t+r}\int_{[\frac{t-r+6l\tau_0}{2},(t-r)\wedge(\be-\tau_0)]\backslash [2(t-r-\frac{\tau_0}{8})-\be,2(t-r+\frac{\tau_0}{8})-\be]}\be-\al d\al d\be\\
&\ge&\int_{t}^{t+r}\int_{[\frac{t-r+6l\tau_0+\tau_0}{2},t-r-\tau_0]}t-r-\tau_0-\al d\al d\be\\
&\ge&\hf(\frac{t-r-6l\tau_0-3\tau_0}{2})^2r.
\eeqa
If we assume $(r,t)\in \Si_{l+1}$, we get
$t-r-6l\tau_0-3\tau_0\sim t-r$ and so is 
 \eqref{form} with $a_{l+1}=a_l+2$, $b_{l+1}=p-1+b_l$ and some $c_{l+1}\in (0, c_l)$.
 
 By induction, 
 with
$a_1=0$ and $b_1=p-1$, it is clear that 
\eqref{form} with $l=j$ could be boosted to 
\eqref{form} with $l=j+1$, as long as $a_j\le b_j$.
As $a_j=2j-2$, $b_j=(p-1)j$, the procedure breaks in finite steps, if $1<p<3$. To be more specific, with $l_0:=\left[\frac{2}{3-p}\right]+1$, we have, for some $c>0$,
\beeqa\label{eq-postibd}
u(t,r)\ge c\vep \frac{r}{(\sinh r)^{\hf}}\left(t+r+\ln\left(\frac{1}{c\vep}\right)\right)^{-l_0(p-1)}(t-r)^{2l_0-2},\ 
\eneqa
for all $(r,t)\in\Si_{l_0}$.
Here, $-l_0(p-1)+2l_0-2=l_0(3-p)-2>0$.

\begin{remark}
When we deal with the iteration of $LF(u)$ by using Lemma \ref{iterforbu} in  \eqref{domain} for large $r$, it shows that \eqref{sdomain} behaves better to maintain the original domain of $LF(u)$ than \eqref{ldomain}, that is, in meaning of abandoned regions, $\tau_0 r$ vs $r^2$.
\end{remark}

\subsubsection{Further improved lower bound}
For this part, we closely follow the idea of John \cite{MR535704} to estimate the critical iterative format of constants.
Equipped with the lower bound \eqref{eq-postibd}, which blows up at infinity, we could exploit the power type nonlinearity to show blow up in finite time.

Let
$A_0=l_0(3-p)-2>0$,  and
$$
Y=\{(\la,\tau): \la>\hf\tau_0, \tau>T, \tau+\la\le T+\tau_0\}\ ,
$$
where $T>\left(\frac{1}{c\vep}\right)^{1/A_0}$ is a constant to be determined later so that
$Y\subset\Si_{l_0}$ (see Figure 3).

Restricted to $Y$,
the lower bound
 \eqref{eq-postibd} tells us that
\beeqa\label{fstep}
u(t,r)\ge \tilde c\vep\ t^{A_0}\ge \tilde c\vep\ T^{A_0},
\eneqa
for some $\tilde c>0$. We shall require $\tilde c\vep\ T^{A_0}>\de_0^{-1}$ so that we could
apply the power type nonlinearity \eqref{eq-nonlinearity2}:
$$F(u)\ge \de_0 |u|^q,\ \forall (r,t)\in Y\ .$$

\begin{figure}[H]
\centering
\includegraphics[width=0.4\textwidth]{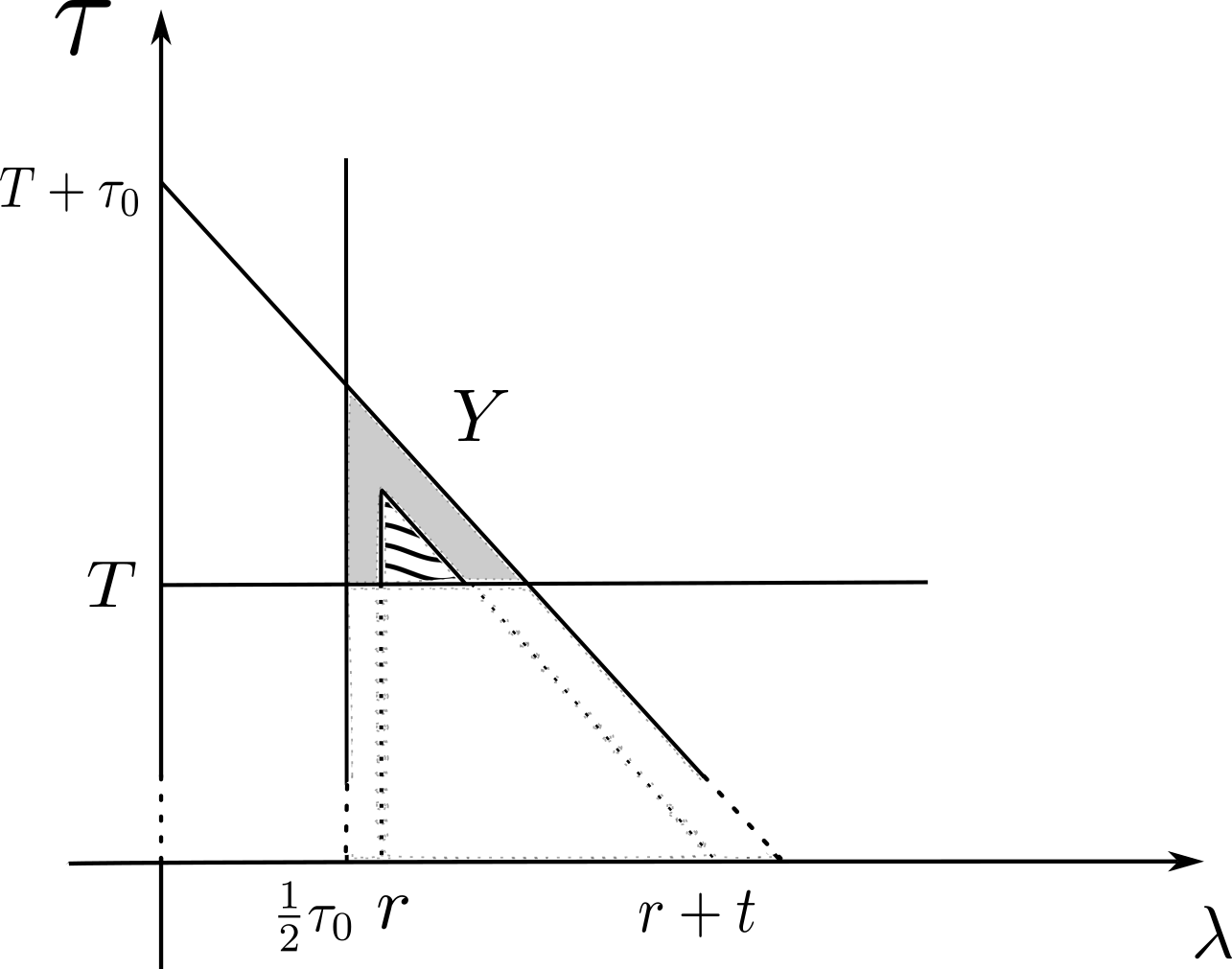}
\centerline{Figure 3}
\end{figure}
As before, we would like to boost \eqref{fstep} to illustrate blow up in finite time.
For such purpose, suppose that we have a lower bound of the following form
\beeqa\label{newform}
u(t,r)\ge D T^{A} (t-T)^{B}\ ,
\eneqa
 for $(r,t)\in Y$, with $A,B\ge0$.
% we claim that for $(t,r)\in Y$, we have the lower bound of the form with respect to $F(u)\ge C|u|^q$ ($C>0$) in the $m$th iteration ($m\ge0$)
Then Lemma \ref{iterforbu} yields for $r>t-\tau$
\begin{eqnarray*}
u(t,r)&\ge&\int_0^tI(t-\tau,r,F(u)(\tau))d\tau\\
&\ge &\frac{C_0}{(\sinh r)^{\hf}}\int_{t-r}^t\int_{r}^{r+t-\tau}F(u)(\tau,\la)(\sinh\la)^{\hf}d\la d\tau\\
&\ge &\frac{C_0\de_0(DT^A)^q}{(\sinh r)^{\hf}}\int_{T}^t\int_{r}^{r+t-\tau}(\tau-T)^{Bq}(\sinh\la)^{\hf}d\la d\tau
\\
&\ge&
C_0\de_0(DT^A)^q
\int_{T}^t(\tau-T)^{Bq}((t-T)-(\tau-T))d\tau
\\
 &=&\frac{ C_0\de_0D^q}{(Bq+1)(Bq+2)}T^{Aq}(t-T)^{Bq+2}\\
 & \ge&
 \frac{ C_0\de_0D^q}{(Bq+2)^2}T^{Aq}(t-T)^{Bq+2}\ .
\end{eqnarray*}
%as $\sinh r\in [r, 2r]$ for $r\in [0,1]$.

As we know \eqref{newform} with $D=\tilde c\vep$, $A=A_0$, $B=0$,
by induction, we have
\eqref{newform} with $D=D_m$, $A=A_m$ and $B=B_m$
for any $m\ge 0$, provided that
$D_0=\tilde c\vep$, $B_0=0$,
and
\beeq\label{m+1}
A_{m+1}=A_m q, B_{m+1}=B_m q+2, D_{m+1}
=%\frac{\de_0  D_m^q }{4 (B_m q+1)(B_m q+2)}\ge
 \frac{C_0\de_0  D_m^q}{ (B_mq+2)^2}=
  \frac{C_0\de_0  D_m^q}{ B_{m+1}^2}
 \ .
\eneq
  Solving \eqref{m+1} yields for $m\ge1$
\beeqa\label{law1}
A_m=A_0 q^m,\  B_m=2\frac{q^m-1}{q-1}\le 2mq^{m-1},\  D_m\ge
\frac{C_0\de_0D_{m-1}^q}{4 m^2q^{2(m-1)}},
\eneqa
and thus
\beqa
D_m\ge \exp\left[q^m\left(\ln D_0-\sum_{j=0}^{m-1}\frac{2\ln(j+1)+2j\ln q-\ln\frac{C_0\de_0}{4}}{q^{j+1}}\right)\right].
\eeqa
Let
\beeqa\label{const}
E=\ln D_0-\sum_{j=0}^{\infty}\frac{2\ln(j+1)+2j\ln q-\ln\frac{C_0\de_0}{4}}{q^{j+1}}\ ,
\eneqa
for which the convergence is ensured by $q>1$,
it follows that, for any $m\ge 1$,
\beeqa\label{law2}
D_m\ge \exp(Eq^m)\ .
\eneqa

Then by \eqref{law1}, \eqref{const}, \eqref{law2}, we have for $(r,t)\in Y$ and sufficiently large $m$
\beqa
u(t,r)\ge \exp\left[q^m\left(E+A_0\ln T+\frac{2}{q-1}\ln(t-T)\right)\right](t-T)^{-\frac{2}{q-1}}.
\eeqa
Let $r=\hf\tau_0$ and $t=T+\hf\tau_0$, 
the term
$E+A_0\ln T+\frac{2}{q-1}\ln(t-T)
$ is positive, for sufficiently large $T$. Then, for such $T$, it follows that $u(T+\hf\tau_0,\hf\tau_0)\to\infty$ as $m\to\infty$, which is the desired contradiction.

\section{Appendix}
This section gathers some technical results we have previously used in Section \ref{pre} and Section \ref{global}.
The first estimate arises from \cite[Lemma D]{MR1443052}.
\begin{lem}\label{partial-tech}
Assume that $a\in C^1[0,\infty)$ satisfies the following conditions:
\begin{item}
\item 1.$a^{\pri}(s)>0$ for $s>0$;\\
\item 2.$\frac{a^{\pri}(t)-a^{\pri}(s)}{\sqrt{a(t)-a(s)}}$ is a monotonically decreasing function of $s$ for any $t>0$ and $0\le s\le t$.\\
\end{item}
Let $\mathbb{M}^n$ be Riemann manifolds with $n$ dimension and $Rv(t,x)=\int_0^t\frac{\hf a^{\pri}(s)}{\sqrt{a(t)-a(s)}}v(s,x)ds$ for each $v\in C(\R_+\times \mathbb{M}^n)$ such that $\pa_tv\in C(\R_+\times \mathbb{M}^n)$, one has
$$|\pa_t(Rv)(t,x)|\le \frac{\hf a^{\pri}(t)}{\sqrt{a(t)-a(0)}}|v(t,x)|+R|\pa_tv|(t,x)\quad (t,x)\in\R_+\times \mathbb{M}^n.$$
\end{lem}

\begin{proof}
Let $L(t,s)=\sqrt{a(t)-a(s)}$ ($0\le s\le t$). From $\pa_sL(t,s)=\frac{-\hf a^{\pri}(s)}{\sqrt{a(t)-a(s)}}$ and $L(t,t)=0$, we deduce $Rv(t,x)=L(t,0)v(0,x)+\int_0^tL(t,s)\pa_sv(s,x)ds$, and
\beqa
\pa_t(Rv)(t,x)&=&\pa_tL(t,0)v(0,x)+\int_0^t\pa_tL(t,s)\pa_sv(s,x)ds\\
&=&\pa_tL(t,0)v(t,x)+\int_0^t[\pa_tL(t,s)-\pa_tL(t,0)]\pa_sv(s,x)ds.
\eeqa
With the fact that $a^{\pri}(0)\ge0$ and
$$\frac{a^{\pri}(t)-a^{\pri}(s)}{\sqrt{a(t)-a(s)}}\le \frac{a^{\pri}(t)-a^{\pri}(0)}{\sqrt{a(t)-a(0)}}\le \frac{a^{\pri}(t)}{\sqrt{a(t)-a(0)}},$$
 we finish the proof by
\beqa
\pa_tL(t,s)-\pa_tL(t,0)=\frac{\hf a^{\pri}(t)}{\sqrt{a(t)-a(s)}}-\frac{\hf a^{\pri}(t)}{\sqrt{a(t)-a(0)}}\le \frac{\hf a^{\pri}(s)}{\sqrt{a(t)-a(s)}}.
\eeqa

\end{proof}

\begin{lem}\label{linear-tech}
Let $a\in C^1(\R)$ be an even function satisfying conditions in Lemma \ref{partial-tech}, $f\in C[0,\infty)$, and
$$W(t,r,f)=\int_0^t\int_{|r-s|}^{r+s}\frac{f(\la)a^{\pri}(s)}{(a(t)-a(s))^{\hf}(a(r+\la)-a(s))^{\hf}(a(s)-a(r-\la))^{\hf}}d\la ds,$$
 one has
\begin{equation*}
|W(t,r,f)|\le\begin{cases}
\pi\int_{r-t}^{r+t}\frac{|f(\la)|}{(a(r+\la)-a(t))^{\hf}}d\la\quad r\ge t\\
\pi\int_0^{t+r}\frac{|f(\la)|}{|a(r+\la)-a(t)|^{\hf}}d\la\quad t>r
\end{cases}.
\end{equation*}
\end{lem}

\begin{proof}
With the help of Beta function, we deduce the conclusion easily.\\
\textbullet $r\ge t$. In this case, we rewrite $W(t,r,f)$ as
\begin{equation*}
\begin{split}
W(t,r,f)&=\int_0^t\int_{r-s}^{r+s}\frac{f(\la)a^{\pri}(s)}{(a(t)-a(s))^{\hf}(a(r+\la)-a(s))^{\hf}(a(s)-a(r-\la))^{\hf}}d\la ds\\
&=\int_{r-t}^{r+t}\int_{|r-\la|}^t\frac{f(\la)a^{\pri}(s)}{(a(t)-a(s))^{\hf}(a(r+\la)-a(s))^{\hf}(a(s)-a(r-\la))^{\hf}} dsd\la,\\
|W(t,r,f)|&\le\int_{r-t}^{r+t}\frac{|f(\la)|}{(a(r+\la)-a(t))^{\hf}}\int_{|r-\la|}^t\frac{a^{\pri}(s)}{(a(t)-a(s))^{\hf}(a(s)-a(r-\la))^{\hf}} dsd\la.
\end{split}
\end{equation*}
With the fact that ($b<c$)
$$\int_{b}^{c}\frac{a^{\pri}(s)}{(a(c)-a(s))^{\hf}(a(s)-a(b))^{\hf}} ds=\int_0^1\nu^{-\hf}(1-\nu)^{-\hf}d\nu=B(\hf,\hf)=\pi\ ,$$
we have done.\\
\textbullet $t>r$. Similarly, we obtain
\begin{equation*}
\begin{split}
W(t,r,f)=&\int_{0}^{t-r}\int_{|r-\la|}^{r+\la}\frac{f(\la)a^{\pri}(s)}{(a(t)-a(s))^{\hf}(a(r+\la)-a(s))^{\hf}(a(s)-a(r-\la))^{\hf}} dsd\la\\
&+\int_{t-r}^{t+r}\int_{|r-\la|}^{t}\frac{f(\la)a^{\pri}(s)}{(a(t)-a(s))^{\hf}(a(r+\la)-a(s))^{\hf}(a(s)-a(r-\la))^{\hf}} dsd\la,\\
|W(t,r,f)|\le&\pi\int_{0}^{t+r}\frac{|f(\la)|}{|a(r+\la)-a(t)|^{\hf}}d\la.
\end{split}
\end{equation*}

\end{proof}
 
% \subsection*{Acknowledgment}
%The first author would like to thank Professor Vladimir Georgiev for proposing the problem of logarithmic nonlinearity,
%during the ``Waseda Workshop on Partial Differential Equations 2019" in Waseda University, where the first author gave a talk on the wave equations on hyperbolic spaces with power type nonlinearity.

\bibliography{bib}
\bibliographystyle{plain}
%对文献姓氏排列
\end{document}